\documentclass[reqno,11pt]{amsart}
\usepackage[T1]{fontenc}
\usepackage{lmodern}
\usepackage[utf8]{inputenc}
\usepackage[english]{babel}
\usepackage[usenames,dvipsnames]{xcolor}
\usepackage{color}
\usepackage{url}
\usepackage[colorlinks=true]{hyperref}
\hypersetup{linkcolor=blue, citecolor=blue, filecolor=black, urlcolor=MidnightBlue}
\usepackage{comment}
\usepackage{geometry}
\usepackage[all]{xy}
\usepackage{enumitem}
\usepackage{booktabs}
\usepackage{slashed}
\usepackage{bbm}
\usepackage{lipsum}

\usepackage{soul}

\usepackage{textcomp}
\usepackage{times}

\usepackage{amsmath}
\usepackage{amssymb,graphicx}
\usepackage{latexsym}
\usepackage{amsfonts}
\usepackage{mathrsfs}

\usepackage{etex}

\numberwithin{equation}{section}

\theoremstyle{plain}
\newtheorem{lemma}{Lemma}[section]
\newtheorem{proposition}[lemma]{Proposition}
\newtheorem{proposition/definition}[lemma]{Proposition/Definition}
\newtheorem{theorem}[lemma]{Theorem}

\theoremstyle{definition}

\newtheorem{remark}[lemma]{Remark}

\DeclareMathOperator{\id}{id}

\newcommand{\calD}{\mathcal{D}}
\newcommand{\calE}{\mathcal{E}}
\newcommand{\calF}{\mathcal{F}}

\newcommand{\calL}{\mathcal{L}}

\newcommand{\calO}{\mathcal{O}}

\newcommand{\calX}{\mathcal{X}}

\newcommand{\bbE}{\mathbb{E}}

\newcommand{\bbQ}{\mathbb{Q}}
\newcommand{\bbR}{\mathbb{R}}
\newcommand{\bbS}{\mathbb{S}}
\newcommand{\bbT}{\mathbb{T}}

\newcommand{\bbZ}{\mathbb{Z}}

\newcommand{\frakB}{\mathfrak{B}}

\newcommand{\frakX}{\mathfrak{X}}

\renewcommand{\theta}{\vartheta}

\renewcommand{\tilde}[1]{\widetilde{#1}}

\title[]{Rigidity of integral coisotropic submanifolds\\ of contact manifolds}

\author{Alfonso Giuseppe Tortorella}\thanks{The author is partially supported by GNSAGA of INdAM}
\address{Dipartimento di Matematica e Informatica ``Ulisse Dini'',  Universit\`a degli Studi di Firenze, Viale Morgagni 67/a, 50134 Florence, Italy.}
\email{\href{mailto:alfonso.tortorella@math.unifi.it}{\underline{alfonso.tortorella@math.unifi.it}}}

\begin{document}

\begin{abstract}
	Unlike Legendrian submanifolds, the deformation problem of co\-iso\-tro\-pic submanifolds can be obstructed.
	Starting from this observation, we single out in the contact setting the special class of \emph{integral coisotropic submanifolds} as the direct generalization of Legendrian submanifolds for what concerns deformation and moduli theory.
	Indeed, being integral coisotropic is proved to be a rigid condition, and moreover the integral coisotropic deformation problem is unobstructed with discrete moduli space.	
\end{abstract}

\maketitle

\section{Introduction}
\label{sec:introduction}

In symplectic geometry, in consequence of the Lagrangian Neighborhood Theorem~\cite{weinstein1971symplectic}, it is well-known that the deformation problem under Hamiltonian equivalence of a compact Lagrangian submanifold $S$ is controlled by its de Rham complex, so that it is unobstructed with local moduli space given by $H^1_{dR}(S)$.
Unlike Lagrangian submanifolds the deformation problem of coisotropic submanifolds is much more involved and hard to manage.
Indeed in~\cite{oh2005deformations} the coisotropic deformation problem is proved to be controlled by an $L_\infty$-algebra, rather than a dg-space.
Moreover the coisotropic deformation problem can be obstructed as explicitly shown in~\cite{zambon2008example}.
However, as pointed out in~\cite{ruan2005deformation}, there is the still interesting class of integral coisotropic submanifolds, whose deformation theory resembles that one of Lagrangian submanifolds.
Indeed the integral coisotropic deformation problem, under Hamiltonian equivalence, is unobstructed, with linear and finite-dimensional local moduli space.

It seems that the contact version of this picture has been only partially unveiled.
Our note aims to fill in these gaps.
In doing this we take advantage of the \emph{line bundle approach} to precontact geometry presented in~\cite{vitagliano2015dirac}.

In contact geometry, in consequence of the Legendrian Neighborhood Theorem~\cite{L1998}, it is well-known that compact Legendrian submanifolds are rigid, i.e.~locally their smooth Legendrian deformations are induced by contact isotopies.
Moreover their deformation problem under contact equivalence is controlled by an acyclic dg-space, so that it is unobstructed with discrete local moduli space (cf., e.g.,~\cite[Section~6]{LOTV}).
As recently shown in~\cite{LOTV}, in the contact setting as well, the coisotropic deformation problem is controlled by an $L_\infty$-algebra, rather than an acyclic dg-space.
In this note we will explicitly construct a first example, in the contact setting, of coisotropic submanifold whose deformation problem is obstructed (Section~\ref{subsec:example_1st_part}).
Further we will single out, in the contact setting, the special class of integral coisotropic submanifolds which behave like Legendrian submanifolds for what concerns deformation and moduli theory.
Indeed we prove that compact integral coisotropic submanifolds are rigid, i.e.~all their smooth integral coisotropic deformations are induced by contact isotopies (Theorem~\ref{prop:rigidity_compact_integral_coisotropic}).
Moreover their deformation problem under contact equivalence is unobstructed, with discrete local moduli space (Proposition~\ref{prop:unobstructedness}).

\subsection{Organization of the paper}
\label{subsec:organization}
Closely following~\cite{vitagliano2015dirac} Section~\ref{sec:precontact_geometry} presents the line bundle approach to precontact geometry and its technical prerequisites: the functorial construction of the Atiyah algebroid of a vector bundle, and the associated Cartan calculus.
In Section~\ref{sec:coisotropic_submanifolds}, after recalling the definition of coisotropic submanifold, we exhibit an explicit example in the contact setting when the coisotropic deformation problem is obstructed.
Section~\ref{sec:integral_coisotropic_submanifolds} contains the main results of this paper.
After introducing integral coisotropic submanifolds we stress their close connection with coisotropic contact reduction.
Finally we prove that compact integral coisotropic submanifolds are rigid (Theorem~\ref{prop:rigidity_compact_integral_coisotropic}), and moreover that their deformation problem under contact equivalence is unobstructed, with discrete local moduli space (Proposition~\ref{prop:unobstructedness}).

\section{A line bundle approach to precontact geometry}
\label{sec:precontact_geometry}

Let $C$ be an hyperplane distribution on a manifold $M$.
Fix a line bundle $L\to M$, and a no-where zero $L$-valued $1$-form $\theta\colon TM\to L$.
Assume that $C$ and $\theta$ are related by the condition $\ker\theta=C$, so that $\theta$ induces the line bundle isomorphism $(TM)/C\simeq L$.
Then the associated curvature form $\omega\in\Gamma(\wedge^2 C^\ast\otimes L)$ is defined by $\omega(X,Y)=\theta([X,Y])$, for all $X,Y\in\Gamma(C)$.
The $1$-form $\theta$, and the hyperplane distribution $C$, are said to be \emph{precontact} (resp.~\emph{contact}) if the vector bundle morphism $\omega^\flat\colon C\to C^\ast\otimes L$, $X\mapsto\omega(X,-)$, has constant rank (resp.~is an isomorphism).
A \emph{(pre)contact manifold} is a manifold $M$ equipped with a \emph{(pre)contact structure} which is equivalently given by a (pre)contact distribution $C$ or a (pre)contact $1$-form $\theta$ on $M$.
Any precontact manifold is endowed with a \emph{characteristic foliation} $\calF$, namely the integral foliation of the involutive distribution $\underline{\smash{K}}_\theta:=\ker\omega^\flat$.
By definition the characteristic foliation of a contact manifold is $0$-dimensional. 
%
A \emph{precontactomorphism} of precontact manifolds $(M,C)\to(M',C')$ is a diffeomorphism $\underline{\smash\varphi}\colon M\to M'$ preserving the precontact distributions, i.e.~such that $(T\underline{\smash\varphi})C=C'$.
Hence in particular the Lie subalgebra $\frakX_C\subset\frakX(M)$ of infinitesimal precontactomorphisms of $(M,C)$ consists of those vector fields $X\in\frakX(M)$ such that $[X,\Gamma(C)]\subset\Gamma(C)$.

\subsection{The Atiyah algebroid and the der-complex of a vector bundle}
\label{subsec:Atiyah_algebroid}
Let $E\to M$ be a vector bundle.	
A \emph{derivation of $E\to M$} is an $\bbR$-linear map $\square\colon \Gamma(E)\to\Gamma(E)$ such that there is a  (unique) $X\in\frakX(M)$, called the \emph{symbol} of $\square$ and also denoted by $\sigma(\square)$, satisfying the  Leibniz rule
\begin{equation*}
	\square(fe)=X(f)e+f\square e,
\end{equation*}
for all $e\in\Gamma(E)$ and $f\in C^\infty(M)$.
The $C^\infty(M)$-module of derivations of $E$, denoted by $\calD E$, is actually the module of sections of a vector bundle $DE\to M$.
For any $x\in M$, the fiber $(DE)_x$ consists of the \emph{derivations of $E$ at $x$}, i.e.~those $\bbR$-linear maps $\delta\colon \Gamma(E)\to E_x$ such that there is a (unique) $\xi\in T_xM$, called the \emph{symbol} of $\delta$ and also denoted by $\sigma(\delta)$, satisfying the  Leibniz rule $\delta(fe)=\xi(f)e_x+f(x)\delta e$, for all $e\in\Gamma(E)$ and $f\in C^\infty(M)$.
The rank of $DE\to M$ is given by $\operatorname{rk}(DE)=\dim M+\operatorname{rk}(E)^2$.
Actually if $x^i$ are local coordinates on $M$ and $\varepsilon^a$ is a local frame of $E\to M$, then an adapted local frame $\square_i,\bbE_a^b$ of $DE\to M$ is defined by
\begin{equation*}
	\square_i(f_c\varepsilon^c)=\frac{\partial f_c}{\partial x^i}\varepsilon^c,\qquad\bbE_a^b(f_c\varepsilon^c)=f_a\varepsilon^b.
\end{equation*}  
The vector bundle $DE\to M$ becomes a \emph{transitive} Lie algebroid, called the \emph{Atiyah algebroid} (or \emph{gauge algebroid}) of $E$, with Lie bracket given by the commutator of derivations $[-,-]$, and anchor given by the symbol map $\sigma\colon DE\to TM$, $\delta\mapsto\sigma(\delta)$.
Moreover the \emph{tautological representation} $\nabla$ of $DE$ in $E$ is defined by $\nabla_\square e=\square e$, for all $\square\in\calD E$, and $ e\in\Gamma(E)$.
These data determine a Cartan calculus on the graded module $\Omega_E^\bullet:=\Gamma(\wedge^\bullet(DE)^\ast\otimes E)$ whose elements are called the \emph{$E$-valued Atiyah forms}.
The structural operations of this Cartan calculus are the following:
\begin{itemize}
	\item the \emph{de Rham differential $d_D$}, i.e.~the degree $1$ derivation of $\Omega_E^\bullet$ such that $(d_D e)(\square)=\square e$, and $(d_D\eta)(\square,\Delta)=\square(\eta(\Delta))-\Delta(\eta(\square))-\eta([\square,\Delta])$, for all $ e\in\Omega^0_E\equiv\Gamma(E)$, $\eta\in\Omega_E^1$, and $\square,\Delta\in\calD E$,
	\item for any $\square\in\calD E$, the \emph{contraction $\iota_\square$}, i.e.~the degree $(-1)$ derivation of $\Omega^\bullet_E$ such that $\iota_\square e=0$, and $(\iota_\square\eta)=\eta(\square)$, for all $ e\in\Omega^0_E\equiv\Gamma(E)$, $\eta\in\Omega_E^1$,	
	\item for any $\square\in\calD E$, the \emph{Lie derivative $\calL_\square$}, i.e.~the degree $0$ derivation of $\Omega^\bullet_E$ such that $(\calL_\square e)=\square e$, and $(\calL_\square\eta)(\Delta)=\square(\eta(\Delta))-\eta([\square,\Delta])$, for all $e\in\Omega^0_E\equiv\Gamma(E)$, $\eta\in\Omega_E^1$, and $\Delta\in\calD E$.
\end{itemize}
Furthermore the above operations of the Cartan calculus are related through the following identities:
\begin{equation*}
	\label{eq:Cartan_identities}
	\begin{gathered}{}
	[d_D,\iota_\square]=\calL_\square,\quad 
	[\iota_\square,\calL_\Delta]=\iota_{[\square,\Delta]},\quad [\calL_\square,\calL_\Delta]=\calL_{[\square,\Delta]},\\
	[d_D,d_D]=[d_D,\calL_\square]=[\iota_\square,\iota_\Delta]=0,
	\end{gathered}
\end{equation*}
for all $\square,\Delta\in\calD E$, where $[-,-]$ is the graded commutator.
Actually $(\Omega_E^\bullet,d_D)$, the de Rham complex of the Atiyah algebroid $DE$ with values in its tautological representation in $E$, also called \emph{der-complex of $E$} (see~\cite{rubtsov1980cohomology}), is acyclic.
Indeed, if $\mathbbm{1}\in\calD E$ denotes the identity map, i.e.~$\mathbbm{1}e=e$, then a contracting homotopy is given by $\iota_{\mathbbm{1}}$, i.e. 
\begin{equation}
	\label{eq:contracting_homotopy}
	[d_D,\iota_{\mathbbm{1}}]=\id.
\end{equation}
\begin{remark}
	\label{rem:derivations_of_line bundles}
	Any derivation of $E$ is actually a first order differential operator from $\Gamma(E)$ to itself, so that $DE$ is a vector subbundle of $(J^1E)^\ast\otimes E$.
	The converse holds iff $E$ is a line bundle.
	Fix a line bundle $L\to M$.
	Since $DL=(J^1L)^\ast\otimes L$, now $\Omega_L^1$ identifies with $\Gamma(J^1L)$ via the vector bundle isomorphism $J^1L\overset{\simeq}{\to}(DL)^\ast\otimes L$ induced by the $L$-valued duality pairing $\langle-,-\rangle\colon DL\otimes J^1L\to L$.
	In view of this, $\Omega_L^0\to\Omega_L^1$, $\lambda\mapsto d_D\lambda$, coincides with the $1$-st jet prolongation $\Gamma(L)\to\Gamma(J^1L),\lambda\mapsto j^1\lambda$.
\end{remark}

Fix vector bundles  $E\to M$ and $E'\to M'$.
Let $\varphi\colon E\to E'$ be a vector bundle morphism, covering $\underline{\smash\varphi}\colon M\to M'$.
Assume that $\varphi$ is \emph{regular}, i.e.~it is fiberwise invertible.
Then the pull-back of sections $\varphi^\ast\colon \Gamma(E')\to\Gamma(E)$, $e'\mapsto\varphi^\ast e'$, is defined by $(\varphi^\ast e')_x=(\varphi|_{E_x})^{-1}e'_{\underline{\smash\varphi}(x)}$, for all $x\in M$.
Further a Lie algebroid morphism $D\varphi\colon DE\to DE'$, covering $\underline{\smash\varphi}$, is defined by $((D\varphi)\delta)e'=\varphi(\delta(\varphi^\ast e'))$, for all $\delta\in DE$, and $e'\in\Gamma(E')$.
Since $\varphi$ and $D\varphi$ are compatible with the tautological representations of $DE$ and $DE'$, they determine a degree zero dg-module morphism $\varphi^\ast\colon (\Omega^\bullet_{E'},d_D)\to(\Omega^\bullet_E,d_D)$ explicitly given by
\begin{equation*}
	(\varphi^\ast\eta')_x(\delta_1,\ldots,\delta_k)=(\varphi|_{E_x})^{-1}(\eta'_{\underline{\smash\varphi}(x)}((D\varphi)\delta_1,\ldots,(D\varphi)\delta_k))
\end{equation*}
for all $\eta'\in\Omega_{E'}^k$, $x\in M$, and $\delta_1,\ldots,\delta_k\in (DE)_x$.
Consequently, if $\square\in\calD E$ and $\square'\in\calD E'$ are \emph{$\varphi$-related}, i.e.~$\varphi^\ast\circ\square'=\square\circ\varphi^\ast$, or equivalently $(D\varphi)\circ\square=\square'\circ\underline{\smash\varphi}$, then $\iota_{\square}\circ\varphi^\ast=\varphi^\ast\circ\iota_{\square'}$, and $\calL_{\square}\circ\varphi^\ast=\varphi^\ast\circ\calL_{\square'}$.

\begin{remark}
	\label{rem:restricted_line_bundle}
	Let $E\to M$ be a vector bundle, and let $S\subset M$ be a submanifold.
	Set $E_S:=E|_S\to S$, and consider the regular vector bundle morphism $i_S\colon E_S\to E$, covering $\underline{i}_S\colon S\to M$, given by the inclusion.
	Then the Lie algebroid morphism $Di_S\colon DE_S\to DE$ identifies $DE_S$ with the Lie subalgebroid $\{\delta\in DE\colon\sigma(\delta)\in TS\}\subset DE$, and $i_S^\ast\colon\Omega^\bullet_E\to\Omega^\bullet_{E_S}$ is a dg-module epimorphism.
\end{remark}

Let $\varphi\colon E\to E'$ be a regular vector bundle morphism covering a surjective submersion $\underline{\smash\varphi}\colon M\to M'$ with connected fibers.
An Atiyah form $\eta\in\Omega_{E}^\bullet$ is said to be \emph{basic} wrt $\varphi$ if there is a (unique) Atiyah form $\eta'\in\Omega_{E'}^\bullet$ such that $\eta=\varphi^\ast\eta'$.
\begin{lemma}
	\label{lem:basic_Atiyah_forms}
	An Atiyah form $\eta\in\Omega_{E}^\bullet$ is basic wrt $\varphi$ iff $\iota_\square\eta=\calL_\square\eta=0$, for all $\square\in\Gamma(\ker D\varphi)\subset\calD E$.	
\end{lemma}
\begin{proof}
	It is a straightforward computation in adapted local frames.
\end{proof}
A smooth path $\square_t$ in $\calD E$ generates the smooth $1$-parameter family $\varphi_t$ of local vector bundle automorphisms of $E\to M$, with $\varphi_0=\id_E$, which is uniquely determined by
\begin{equation}
	\label{eq:Lie_derivative_formula_0}
	\frac{d}{dt}\varphi_t^\ast e=\varphi_t^\ast(\square_t e),\ \text{for all}\ e\in\Gamma(E).
\end{equation}
Notice that $\underline{\smash\varphi}_t$, the smooth $1$-parameter family of local diffeomorphisms of $M$ covered by $\varphi_t$, is the flow of the time-dependent vector field on $M$ given by $\sigma(\square_t)$.
Moreover~\eqref{eq:Lie_derivative_formula_0} extends into the following Lie derivative formula
\begin{equation}
	\label{eq:Lie_derivative_formula}
	\frac{d}{dt}\varphi_t^\ast\eta=\varphi_t^\ast(\calL_{\square_t}\eta),\ \text{for all}\ \eta\in\Omega^\bullet_E.
\end{equation}
Conversely, for any smooth $1$-parameter family $\varphi_t$ of local vector bundle automorphisms of $E\to M$, with $\varphi_0=\id_E$, its infinitesimal generator is the time-dependent derivation $\square_t$ of $E\to M$ uniquely determined by~\eqref{eq:Lie_derivative_formula_0}.

\begin{remark}
	\label{rem:homogeneous_deRham_complex}
	The next Section~\ref{subsec:presymplectization} summarizes the line bundle approach to precontact geometry.
	In~\cite{vitagliano2015dirac} this approach was inspired by the ``presymplectization trick'' in view of the existing connection between the der-complex of $E$ and the homogeneous de Rham complex of $E^\ast$ which we are now going to outline.
	Let $E\to M$ be a vector bundle.
	Denote by $\calE\in\frakX(E^\ast)$ the Euler vector field on the total space of the dual vector bundle.
	A canonical $C^\infty(M)$-module embedding $\Gamma(E)\to C^\infty(E^\ast)$, $e\mapsto\tilde e$, is defined by $\phi^\ast\tilde e=\langle\phi,e\rangle$ for all $\phi\in\Gamma(E^\ast)$ and $e\in\Gamma(E)$, where $\langle-,-\rangle$ denotes the duality pairing.
	The image of $\Gamma(E)\to C^\infty(E^\ast)$ consists of those functions $f$ that are \emph{degree one homogeneous}, i.e.~$\calE(f)=f$.
	A vector field $X\in\frakX(E^\ast)$ is said \emph{linear}, or \emph{degree zero homogeneous}, if $[\calE,X]=0$. 
	The space of linear vector fields, denoted by $\frakX_{\textnormal{lin}}(E^\ast)$, is both a $C^\infty(M)$-submodule and a Lie subalgebra of $\frakX(E^\ast)$.
	There is a canonical, $C^\infty(M)$-linear, Lie algebra isomorphism $\calD E\to\frakX_{\textnormal{lin}}(E^\ast)$, $\square\mapsto\tilde{\square}$, defined by $\tilde{\square}\tilde{e}=\tilde{(\square e)}$, for all $\square\in\calD E$, and $e\in\Gamma(E)$.
	Denote by $\Omega_{\textnormal{h}}^\bullet(E^\ast)$ the subcomplex of $(\Omega^\bullet(E^\ast),d_{dR})$ formed by those differential forms $\alpha$ that are \emph{degree one homogeneous}, i.e.~$\calL_{\calE}\alpha=\alpha$.
	Notice that $(\Omega_{\textnormal{h}}^\bullet(E^\ast),d_{dR})$, also called the \emph{homogeneous de Rham complex of $E^\ast$}, is acyclic with a contracting homotopy given by $\iota_{\calE}$.
	Finally there is a canonical (degree zero) dg-module monomorphism $(\Omega_{\textnormal{h}}^\bullet(E^\ast),d_{dR})\to(\Omega_E^\bullet,d_D)$, $\alpha\mapsto\alpha_D$, which intertwines $\iota_{\calE}$ with $\iota_{\mathbbm{1}}$, and is defined by $\tilde{\alpha_D(\square_1,\ldots,\square_k)}=\alpha(\tilde{\square_1},\ldots,\tilde{\square_k})$, for all $\alpha\in\Omega_{\textnormal{h}}^k(E^\ast)$ and $\square_1,\ldots,\square_k\in\calD E$.
	In particular, $\Omega_{\textnormal{h}}^\bullet(E^\ast)\to\Omega_E^\bullet$ is an isomorphism if and only if $E\to M$ is a line bundle.
	For more details about the homogeneous de Rham complex see~\cite[Sec.~4]{vitagliano2015multicontact}.
\end{remark}

\subsection{Precontact geometry as presymplectic geometry on the Atiyah algebroid of a line bundle}
\label{subsec:presymplectization}
Let $L\to M$ be a line bundle, and let $\varpi\in\Omega^2_L$ be such that $\iota_{\mathbbm{1}}\varpi$ is no-where zero.
This $L$-valued Atiyah $2$-form $\varpi$ is said to be \emph{presymplectic} (resp.~\emph{symplectic}) if $\varpi$ is $d_D$-closed, and the vector bundle morphism $\varpi^\flat\colon DL\to(DL)^\ast\otimes L$, $\square\mapsto \varpi(\square,-)$, has constant rank (resp.~is an isomorphism).
Any $L$-valued presymplectic Atiyah form $\varpi$ determines the Lie subalgebroid $K_\varpi:=\ker\varpi^\flat\subset DL$; clearly if $\varpi$ is symplectic then $K_\varpi=0$. 

	\begin{proposition}[{\cite[Proposition~3.3]{vitagliano2015dirac}}]
		\label{prop:Atiyah_presymplectic_forms}
		For any line bundle $L\to M$, the relation $\theta\circ\sigma=\iota_{\mathbbm{1}}\varpi$ establishes a one-to-one correspondence between $L$-valued (pre)contact forms $\theta$ and the $L$-valued (pre)symplectic Atiyah forms $\varpi$.
		Moreover if $\theta$ and $\varpi$ correspond to each other, then the symbol map $\sigma\colon DL\to TM$ induces a Lie algebroid isomorphism from $K_\varpi\subset DL$ to $\underline{\smash K}_\theta\subset TM$.
	\end{proposition}

\begin{proof}
	It follows from~\eqref{eq:contracting_homotopy} and the fact that $\sigma:DL\to TM$ is surjective with $\ker\sigma=\langle\mathbbm{1}\rangle$.
\end{proof}

\begin{remark}
	\label{rem:pre-symplectization_trick}
	The line bundle approach to precontact geometry, summarized in Proposition~\ref{prop:Atiyah_presymplectic_forms}, was inspired by the presymplectization construction as we are going to briefly recall.
	Let $\theta$ be an $L$-valued contact (resp.~precontact) form on $M$ and let $\varpi$ be the corresponding $L$-valued symplectic (resp.~presymplectic) Atiyah form.
	In view of Remark~\ref{rem:homogeneous_deRham_complex}, the dg-module isomorphism $(\Omega_{\textnormal{h}}^\bullet(L^\ast),d_{dR})\to(\Omega_L^\bullet,d_D)$ identifies $\varpi$ with a degree one homogeneous closed $2$-form $\tilde{\varpi}$ on $L^\ast$.
	It is easy to see, in adapted local coordinates, that $\tilde{\varpi}$ has maximal (resp.~constant) rank over $\tilde{M}:=L^\ast\setminus{\bf 0}_M$, where ${\bf 0}_M$ denotes the image of the zero section.
	Indeed the (pre)symplectic form $\tilde{\varpi}|_{\tilde{M}}$ is the \emph{(pre)symplectization} of the (pre)contact form $\theta$.
\end{remark}


Recall that a \emph{Jacobi structure} on a line bundle $L\to M$ is a Lie bracket $\{-,-\}$ on $\Gamma(L)$ which is a derivation of $L$ in both entries.
Notice that a skew-symmetric bi-derivation $\{-,-\}:\Gamma(L)\times\Gamma(L)\to\Gamma(L)$ identifies with the section $J\in\Gamma(\wedge^2(J^1L)^\ast\otimes L)$ such that $J(j^1\lambda,j^1\mu)=\{\lambda,\mu\}$, for all $\lambda,\mu\in\Gamma(L)$.
A Jacobi structure $J=\{-,-\}$ on $L\to M$ is called \emph{non-degenerate} if $J^\sharp\colon J^1L\to DL$, $\eta\mapsto J(\eta,-)$, is a vector bundle isomorphism.
For more details about Jacobi structures see, e.g.,~\cite{LOTV,LTV,T2017}.

	\begin{proposition}[{\cite[Proposition~3.6]{vitagliano2015dirac}}]
		\label{prop:non-degenerate_Jacobi_structure}
		For any line bundle $L\to M$, the relation $\varpi^\flat=(J^\sharp)^{-1}$ establishes a one-to-one correspondence between $L$-valued symplectic Atiyah forms $\varpi$ and non-degenerate Jacobi structures $J$ on $L\to M$.
		Moreover if $\varpi$ and $J$ correspond to each other, then there is a Lie algebra isomorphism $\calX:\Gamma(L)\to\frakX_C$, $\lambda\mapsto\calX_\lambda:=\sigma(J^\sharp(j^1\lambda))$, where $C:=\ker\theta$, and $\theta$ is the $L$-valued contact form corresponding to $\varpi$ according to Proposition~\ref{prop:Atiyah_presymplectic_forms}.
	\end{proposition}

\begin{proof}
	Fix a skew-symmetric bi-derivation $J=\{-,-\}:\Gamma(L)\times\Gamma(L)\to\Gamma(L)$ and an $L$-valued Atiyah $2$-form $\varpi$ such that $J^\sharp:J^1L\to DL$ and $\omega^\flat:DL\to J^1L$ are inverses of each other.
	Set $\Delta_\lambda:=\{\lambda,-\}=J^\sharp(j^1\lambda)\in\calD L$, for all $\lambda\in\Gamma(L)$.
	Then a straightforward computation shows that
	\begin{equation*}
		(d_D\varpi)(\Delta_\lambda,\Delta_\mu,\Delta_\nu)=\{\lambda,\{\mu,\nu\}\}+\{\mu,\{\nu,\lambda\}\}+\{\nu,\{\lambda,\mu\}\},
	\end{equation*}
	for all $\lambda,\mu,\nu\in\Gamma(L)$.
	Hence $\varpi$ is $d_D$-closed iff $\{-,-\}$ satisfies the Jacobi identity.
\end{proof}

\section{Coisotropic submanifolds of contact manifolds}
\label{sec:coisotropic_submanifolds}

Let $\theta$ be an $L$-valued contact form on $M$, with associated contact distribution $C$ and curvature form $\omega$.
Denote by $\varpi$ the $L$-valued symplectic Atiyah form corresponding to $\theta$ according to Proposition~\ref{prop:Atiyah_presymplectic_forms}.
Fix a submanifold $S\subset M$, and set $C_S:=C\cap TS$.
Assume that the distribution $C_S$ on $S$ has constant rank.
The submanifold $S$ is said to be \emph{coisotropic} in $(M,C)$ if $C_S{}^{\perp_\omega}\subset C_S$, where $\perp_\omega$ denotes the orthogonal complement in $C$ wrt $\omega$.
Any coisotropic submanifold $S$ is endowed with the \emph{characteristic foliation} $\calF$ such that $T\calF=C_S{}^{\perp_\omega}$.
\begin{remark}
	Having assumed that $C_S$ has constant rank, we only consider in the contact setting coisotropic submanifolds whose characteristic distribution is non-singular.
	Additionally, in view of this initial assumption, only two possibilities exist for a coisotropic submanifold $S$:
	\begin{enumerate}
		\item\label{enumitem:Legendrian_case}
		$S$ is tangent to $C$, i.e.~$TS\subset C$, or equivalently $C_S=TS$,
		\item\label{enumitem:regular_case}
		$S$ is transverse to $C$, i.e.~$(TM)|_S=C|_S + TS$, or equivalently $C_S$ is an hyperplane distribution on $S$.
	\end{enumerate}
	The first case actually describes the Legendrian submanifolds $S$ of $(M,C)$, i.e.~the maximally isotropic submanifolds $S$ of $(M,C)$, so that $T\calF=TS$.
	The second case describes exactly those coisotropic submanifolds $S$ of $(M,C)$ that are called \emph{regular} (cf.~\cite[Definition~5.10]{LOTV}).
\end{remark}

The next proposition provides an equivalent characterization of regular coisotropic submanifolds.

\begin{proposition}
	\label{prop:regular_coisotropic_submanifolds}
	Assume that $C_S$ has constant rank.
	Then the following conditions are equivalent:	
	\begin{enumerate}
		\item
		\label{enumitem:regular_coisotropic_submanifolds_1}
		$S$ is a regular coisotropic submanifold of $(M,C)$,
		\item
		\label{enumitem:regular_coisotropic_submanifolds_2}
		$\theta_S:=\theta\circ T\underline{\smash{i}}_S=\theta|_{TS}$ is an $L_S$-valued precontact form, with $\underline{\smash K}_{\theta_S}=C_S{}^{\perp\omega}$,
		\item
		\label{enumitem:regular_coisotropic_submanifolds_3}
		$\varpi_S:=i_S^\ast\varpi$ is a $L_S$-valued presymplectic Atiyah form, with $K_{\varpi_S}=(DL_S)^{\perp\varpi}$, where $\perp_{\varpi}$ denotes the othogonal complement in $DL$ wrt $\varpi$.
	\end{enumerate}
	If equivalent conditions~\eqref{enumitem:regular_coisotropic_submanifolds_1}--\eqref{enumitem:regular_coisotropic_submanifolds_3} hold, then $\theta_S$ and $\varpi_S$ correspond to each other within Proposition~\ref{prop:Atiyah_presymplectic_forms} and make $S$ into a precontact manifold with curvature form $\omega|_{C_S}$ and characteristic distribution $C_S{}^{\perp_\omega}$.
\end{proposition}

\begin{proof}
	(1)$\Leftrightarrow$(2).
	It is straightforward from the definition of regular coisotropic submanifold.
	
	(2)$\Leftrightarrow$(3).
	$L_S$-valued $1$-form $\theta_S$ and $L_S$-valued Atiyah $2$-form $\varpi_S$ satisfy the relation $\theta_S\circ\sigma=\iota_{\mathbbm{1}}\varpi_S$ because of their very definition.
	Hence, according to Proposition~\ref{prop:Atiyah_presymplectic_forms}, $\theta_S$ is a precontact form iff $\varpi_S$ is a presymplectic Atiyah form.
	The remaining part of the statement follows immediately.
\end{proof}

\begin{remark}
	Every regular coisotropic submanifold inherits a precontact structure.
	The Coisotropic Neighborhood Theorem~\cite{L1998} states that also the converse holds.
	Indeed for any precontact manifold $(S,C_S)$ there exists (unique up to local contactomorphisms) a \emph{contact thickening}, i.e.~an embedding as regular coisotropic submanifold into a contact manifold.
	In particular, a contact thickening is associated with any choice of a distribution $G$ on $S$ complementary to $T\calF$ (see, e.g.,~\cite[Section~5.3]{LOTV}).
\end{remark}

Let $S$ be a coisotropic submanifold of a contact manifold $(M,C)$.
A \emph{smooth coisotropic deformation} of $S$ is a smooth $1$-parameter family of embeddings $\underline{\smash\varphi}_t\colon S\to M$, with $\underline{\smash\varphi}_0=\id_S$, such that $\underline{\smash\varphi}_t(S)$ is coisotropic in $(M,C)$.
Two smooth coisotropic deformations $\underline{\smash\varphi}_t'$ and $\underline{\smash\varphi}_t''$ are identified if $\underline{\smash\varphi}_t'(S)=\underline{\smash\varphi}_t''(S)$.
Notice that any \emph{contact isotopy} of $(M,C)$, i.e.~any smooth $1$-parameter family $\underline{\smash{\psi}}_t$ of contactomorphisms of $(M,C)$, with $\underline{\smash{\psi}}_0=\id_M$, induces the smooth coisotropic deformation of $S$ defined by $\underline{\smash{\varphi}}_t:=\underline{\smash{\psi}}_t|_S$.
A smooth coisotropic deformation is called \emph{trivial} if it simply consists of diffeomorphisms of $S$.
Being interested only in small deformations of $S$, it is possible to assume to work within a tubular neighborhood $\tau:M\to S$ of $S$ in $M$.
In this restricted setting, a section $s:S\to M$ of $\tau$ is an \emph{infinitesimal coisotropic deformation} of $S$ if the image of $\varepsilon s$ is coisotropic up to infinitesimal $\calO(\varepsilon^2)$, where $\varepsilon$ is a formal parameter.

\subsection{An example of obstructed coisotropic submanifold}
\label{subsec:example_1st_part}

In this one and the following Section~\ref{subsec:example_2nd_part} we present what can be seen as the contact analogue of Zambon's example~\cite{zambon2008example}.
In doing this we closely follow the original approach in the symplectic case.

Let $(M,C)$ be the contact manifold, where $M:=\bbT^5\times\bbR^2$, and $C$ is the kernel of the contact $1$-form $\theta:=\sin x_1 dx_2+\cos x_1 dx_3+y_4dx_4+y_5dx_5$, with $(x_1,\ldots,x_5)$ and $(y_4,y_5)$ denoting the standard coordinates on $\bbT^5$ and $\bbR^2$ respectively.
It is easy to see that $S:=\bbT^5\times\{0\}\simeq\bbT^5$ is a regular coisotropic submanifold, and its precontact structure $C_S$ is the kernel of the precontact $1$-form $\theta_S:=\sin x_1 dx_2+\cos x_1 dx_3$.
Denote by $\calF$ the characteristic foliation of $(S,C_S)$.

The global frame $d_\calF x_4:=(dx_4)|_{T\calF}$, $d_\calF x_5:=(dx_5)|_{T\calF}$ identifies $T^\ast\calF\to S$ with the trivial vector bundle $\tau\colon M=\bbT^5\times\bbR^2\to S=\bbT^5$.
Under this identification, $(M,C)$ identifies with $(T^\ast\calF,\ker(\tau^\ast\theta_S+\theta_G))$, i.e.~the contact thickening of $S$ associated with $G=\operatorname{span}\{\partial/\partial x_1,\partial/\partial x_2,\partial/\partial x_3\}$.

The small coisotropic deformations of $S$ in $(M,C)$ are given by the \emph{coisotropic sections} of $\tau:M\to S$, i.e.~those sections whose image is coisotropic wrt $C$.
A straightforward computation shows that a section $s=fd_\calF x_4+gd_\calF x_5$ is coisotropic iff $f,g\in C^\infty(\bbT^5)$ satisfy the following non-linear first-order pde
\begin{equation}
	\label{eq:contact_coisotropicity}
	\frac{\partial f}{\partial x_1}X(g)-\frac{\partial g}{\partial x_1}X(f)=\frac{\partial g}{\partial x_4}-\frac{\partial f}{\partial x_5}+gY(f)-fY(g),
\end{equation}
where $X:=\cos x_1\frac{\partial}{\partial x_2} -\sin x_1\frac{\partial}{\partial x_3}$, and $Y:=\sin x_1\frac{\partial}{\partial x_2} +\cos x_1\frac{\partial}{\partial x_3}$.
\begin{remark}
	Characterization~\eqref{eq:contact_coisotropicity} of coisotropic sections agrees with the results obtained via the BFV-complex attached to $S$ (cf.~\cite[Sec.~6.4]{T2017}).
	The reader is refereed to~\cite{LTV} for more details about the BFV-complex attached to a coisotropic submanifold, in the more general setting of Jacobi manifolds, and its role in the coisotropic deformation problem. 
\end{remark}

Linearizing~\eqref{eq:contact_coisotropicity}, we see that the infinitesimal coisotropic deformations of $S$ in $(M,C)$ are described exactly by those sections $s=fd_\calF x_4+gd_\calF x_5$ such that $f,g\in C^\infty(\bbT^5)$ satisfy the following linear first-order pde
\begin{equation}
	\label{eq:contact_infinitesimal_coisotropicity}
	\frac{\partial g}{\partial x_4}-\frac{\partial f}{\partial x_5}=0.
\end{equation}
Let $s=fd_\calF x_4+gd_\calF x_5$ be an infinitesimal coisotropic deformation.
Assume that $s$ is \emph{prolongable}, i.e.~there is a smooth coisotropic deformation $s_t$ such that $s=\left.\frac{d}{dt}\right|_{t=0}s_t$.
Then, integrating~\eqref{eq:contact_coisotropicity} over $(x_4,x_5)\in\bbT^2$, we get the following necessary condition for the prolongability of $s$:
\begin{equation}
	\label{eq:contact_kuranishi}
	0=\int\limits_0^{2\pi}\int\limits_0^{2\pi}\left[\frac{\partial f}{\partial x_1}X(g)-\frac{\partial g}{\partial x_1}X(f)+fY(g)-gY(f)\right]dx_4 dx_5.
\end{equation}

\begin{remark}
	Characterization~\eqref{eq:contact_infinitesimal_coisotropicity} of infinitesimal coisotropic deformations amounts to the cocycle condition on $s=f d_\calF x_4+g d_\calF x_5$ within the de Rham complex of the Lie algebroid $T\calF\to S$.
	Furthermore obstruction~\eqref{eq:contact_kuranishi} to the prolongability of $s$ is nothing but the vanishing at $[s]$ of the Kuranishi map $\operatorname{Kur}\colon H^1(\calF)\to H^2(\calF)$.
	Hence~\eqref{eq:contact_infinitesimal_coisotropicity} and~\eqref{eq:contact_kuranishi} agree with the results obtained via the $L_\infty$-algebra attached to $S$ (cf.~\cite[Sec.~4.8]{T2017}).
	See also~\cite{LOTV} for details about the $L_\infty$-algebra attached to a coisotropic submanifold, in the more general setting of Jacobi manifolds, and its role in the coisotropic deformation problem.
\end{remark}

Set $s=fd_\calF x_4+gd_\calF x_5$, with $f:=\cos x_2$, and $g:=\sin x_2$.
Clearly $s$ satisfies~\eqref{eq:contact_infinitesimal_coisotropicity}, and so it is an infinitesimal coisotropic deformation of $S$.
However $s$ is not prolongable to a smooth coisotropic deformation because it fails to fulfill~\eqref{eq:contact_kuranishi}.
Indeed, in this case, the rhs of~\eqref{eq:contact_kuranishi} is equal to
\begin{equation*}
	\label{eq:contact_obstructed_infinitesimal_coisotropic_deformation}
	\int\limits_0^{2\pi}\int\limits_0^{2\pi}\left[\frac{\partial f}{\partial x_1}X(g)-\frac{\partial g}{\partial x_1}X(f)+fY(g)-gY(f)\right]dx_4 dx_5=(2\pi)^2\sin x_1\neq 0.
\end{equation*}
The above discussion is summarized by the next proposition.

\begin{proposition}
	\label{prop:obstructed_coisotropic_contact_case}
	The coisotropic deformation problem of $S$ is obstructed, i.e.~there exists an infinitesimal coisotropic deformation of $S$ which cannot be prolonged to a smooth coisotropic deformation of $S$.
\end{proposition}

\section{Integral coisotropic submanifolds of a contact manifold}
\label{sec:integral_coisotropic_submanifolds}

A coisotropic submanifold $S$ of a contact manifold $(M,C)$ is said to be \emph{integral} if the characteristic foliation $\calF$ of $S$ is simple, i.e.~there is a (unique) structure of smooth manifold on the characteristic leaf space $S/\calF$ such that the quotient map $S\to S/\calF$ is a surjective submersion.
In other words there is a surjective submersion $\underline{\smash{\pi}}:S\to B$, with connected fibers, such that $T\calF=\ker T\underline{\smash{\pi}}$.
This notion of being integral applies verbatim to precontact manifolds as well.

\begin{remark}
	\label{rem:integral_coisotropic_submanifolds}
	All Legendrian submanifolds are integral.
	Indeed a connected Legendrian submanifold $S$ has only one characteristic leaf: $S$ itself.
	On the contrary for a regular coisotropic submanifold the condition of being integral is generically non-trivial and it depends only on its inherited precontact structure.
	Actually a regular coisotropic submanifold $S$ of $(M,C)$ is integral iff  $S$ is an integral precontact manifold when equipped with the precontact structure inherited from $(M,C)$.
	The following Proposition~\ref{prop:integral_precontact_manifold} leads to an equivalent characterization of integral regular coisotropic submanifolds.
\end{remark}

Let $\theta$ be an $L$-valued contact form on $M$, with associated contact distribution $C$ and curvature form $\omega$.
Denote by $\varpi$ the $L$-valued symplectic Atiyah form corresponding to $\theta$ within Proposition~\ref{prop:Atiyah_presymplectic_forms}.

\begin{proposition}[Contact Reduction]
	\label{prop:integral_precontact_manifold}
	Let $S$ be a regular coisotropic submanifold of $(M,C)$.
	For any surjective submersion $\underline{\smash\pi}\colon S\to B$, with connected fibers, the following conditions are equivalent:
	\begin{enumerate}
		\item\label{enumitem:integral_precontact_manifold_1}
		$S$ is integral with $T\calF=\ker T\underline{\smash{\pi}}$,
		\item\label{enumitem:integral_precontact_manifold_2}
		there is a (unique) contact distribution $C_B$ on $B$ such that $C_S=(T\underline{{\smash\pi}})^{-1}C_B$,
		\item\label{enumitem:integral_precontact_manifold_3}
		there is a regular line bundle morphism $\pi:L_S\to L_B$, covering $\underline{\smash{\pi}}:S\to B$, and a (unique) $L_B$-valued contact form $\theta_B$ such that $\pi\circ\theta_S=\theta_B\circ T\underline{\smash\pi}$,
		\item\label{enumitem:integral_precontact_manifold_4}
		there is a regular line bundle morphism $\pi:L_S\to L_B$, covering $\underline{\smash{\pi}}:S\to B$, and a (unique) $L_B$-valued symplectic Atiyah form $\varpi_B$ such that $\varpi_S=\pi^\ast\varpi_B$.
	\end{enumerate}
	If the equivalent conditions~\eqref{enumitem:integral_precontact_manifold_1}--\eqref{enumitem:integral_precontact_manifold_4} hold, then $B$ is called the \emph{reduced contact manifold} of $S$, with \emph{contact reduction} performed via $\underline{\smash\pi}$ (or $\pi$).
\end{proposition}

\begin{proof}
	(1)$\Leftrightarrow$(2).
	If $\ker T\underline{\smash{\pi}}\subset T\calF$, then $C_B:=(T\underline{\smash{\pi}})C_S$ is an hyperplane distribution on $B$, with $C_S=(T\underline{\smash{\pi}})^{-1}C_B$.
	Set $L_B:=(TB)/C_B$ and denote by $\theta_B$ the quotient bundle map $TB\to L_B$.
	Then the curvature form $\omega_B$ associated with $\theta_B$ is completely determined by
	\begin{equation}
		\label{eq:relation_curvature_forms}
		\omega_B((T\underline{\smash{\pi}})\xi_1,(T\underline{\smash{\pi}})\xi_2)=\pi(\omega_S(\xi_1,\xi_2)),
	\end{equation}
	for all $x\in S$ and $\xi_1,\xi_2\in (C_S)_x$, where $\pi:L_S\to L_B$ is the unique regular line bundle morphism such that $\pi\circ\theta_S=\theta_B\circ T\underline{\smash\pi}$.
	From~\eqref{eq:relation_curvature_forms} it follows that $T\calF\subset\ker T\underline{\smash{\pi}}$ iff $\ker\omega_B^\flat=0$, i.e.~$C_B$ is a contact distribution.
	
	(2)$\Leftrightarrow$(3).
	We concentrate only on (2)$\Rightarrow$(3) because the converse is obvious.
	Set $L_B:=(TB)/C_B$ and denote by $\theta_B$ the $L_B$-valued contact form given by the quotient bundle map $TB\to (TB)/C_B$.
	Since $C_S=(T\underline{\smash{\pi}})^{-1}C_B$, there is a unique regular line bundle morphism $\pi:L_S\to L_B$ such that $\pi\circ\theta_S=\theta_B\circ T\underline{\smash\pi}$.
	
	(3)$\Leftrightarrow$(4).
	Let $\pi:L_S\to L_B$ be a regular line bundle morphism covering $\underline{\smash{\pi}}:S\to B$.
	Fix an $L_B$-valued contact form $\theta_B$ and an $L_B$-valued symplectic Atiyah form $\varpi_B$.
	Assume that $\theta_B$ and $\varpi_B$ correspond to each other within Proposition~\ref{prop:Atiyah_presymplectic_forms}, i.e.~$\theta_B\circ\sigma=\iota_{\mathbbm{1}}\varpi_B$.
	Now, using~\eqref{eq:contracting_homotopy}, it is just an easy computation to check that $\pi\circ\theta_S=\theta_B\circ T\underline{\smash{\pi}}$ iff $\varpi_S=\pi^\ast\varpi_B$.
\end{proof}

Let $S$ be an integral coisotropic submanifold of a contact manifold $(M,C)$, with characteristic foliation $\calF$.
A smooth coisotropic deformation $\underline{\smash\varphi}_t$ of $S$ is called \emph{integral} if the coisotropic submanifold $S_t:=\underline{\smash\varphi}_t(S)$ is integral, with characteristic foliation $\calF_t$, and additionally $\underline{\smash\varphi}_t$ induces a diffeomorphism from the characteristic leaf space $S/\calF$ of $S$ to the characteristic leaf space $S_t/\calF_t$ of $S_t$.
Hence, in particular, for a smooth integral coisotropic deformation $\underline{\smash\varphi}_t$, all the characteristic leaf spaces $S_t/\calF_t$ are diffeomorphic to each other.
Notice that, if a smooth coisotropic deformation $\underline{\smash{\varphi}}_t$ of $S$ is induced by a contact isotopy $\underline{\smash{\psi}}_t$, then $\underline{\smash{\varphi}}_t$ is integral.
On the contrary, in general, a smooth coisotropic deformation of an integral regular coisotropic submanifold can fail to be integral as shown by an explicit example in Section~\ref{subsec:example_2nd_part} (see also Proposition~\ref{prop:contact_integral_coisotropic_not_invariant}).
As already mentioned, being interested only in small deformations of $S$, we can assume to work within a tubular neighborhood $\tau:M\to S$ of $S$ in $M$.
In such setting, a section $s:S\to M$ of $\tau$ is an \emph{infinitesimal integral coisotropic deformation} of $S$ if the image of $\varepsilon s$ is integral coisotropic up to infinitesimal $\calO(\varepsilon^2)$, where $\varepsilon$ is a formal parameter.

\subsection{An example of integral coisotropic submanifold unstable under small coisotropic deformations}
\label{subsec:example_2nd_part}

The precontact manifold $(S,C_S)$ of Section~\ref{subsec:example_1st_part} is integral, with reduced contact manifold $(B,C_B)$ and contact reduction performed via $\underline{\smash\pi}\colon S\to B$, where $\underline{\smash \pi}$ is the projection of $S=\bbT^3\times\bbT^2$ on $B=\bbT^3$ and $C_B=\operatorname{span}\{\frac{\partial}{\partial x_1},Y\}$ is the kernel of the contact $1$-form $\theta_B:=\sin x_1 dx_2+\cos x_1 dx_3$.
Hence in particular the characteristic foliation $\calF$ of $(S,C_S)$ is the fibration in $2$-tori provided by $\underline{\smash \pi}$.

Let $s=f d_\calF x_4+g d_\calF x_5$ be an arbitrary coisotropic section.
The image $S'$ of the coisotropic section $s$ and its inherited precontact $1$-form $\theta|_{TS'}$ identify with $S$ and the precontact $1$-form $s^\ast\theta=\cos x_1 dx_2+\sin x_1 dx_3+fdx_4+gdx_5$ respectively.
The characteristic foliation $\calF'$ of $(S,\ker(s^\ast\theta))$ is the integral foliation of the rank $2$ distribution given by
\begin{equation}
	\label{eq:contact_characteristic_distribution}
	T\calF'=\operatorname{span}\left\{X(f)\frac{\partial}{\partial x_1}-\frac{\partial f}{\partial x_1}X+\frac{\partial}{\partial x_4}-fY,\ X(g)\frac{\partial}{\partial x_1}-\frac{\partial g}{\partial x_1}X+\frac{\partial}{\partial x_5}-gY\right\}.
\end{equation}
As a consequence, each leaf $\calL$ of $\calF'$ is bi-dimensional and transverse to the fibers of the projection $p\colon \bbT^5\to\bbT^2,\ (x_1,\ldots,x_5)\mapsto(x_4,x_5)$, so that, by a theorem of Ehresmann~\cite[Ch.~V, Sec.~2]{ehresmann}, $p|_{\calL}\colon\calL\to\bbT^2$ is a covering map.
Hence an arbitrary leaf of $\calF'$ can only be diffeomorphic to $\bbR^2$, $\bbR\times\bbT^1$ or $\bbT^2$.

Consider the smooth $1$-parameter family of coisotropic sections $s_t:=t\sin x_1 d_\calF x_4$, with $t\in\bbR$.
According to~\eqref{eq:contact_characteristic_distribution}, the characteristic foliation $\calF'_t$ of $(S,\ker(s_t^\ast\theta))$ is determined by
\begin{equation}
	\label{eq:contact_characteristic_distribution_bis}
	T\calF'_t=\operatorname{span}\left\{\frac{\partial}{\partial x_4}-t\frac{\partial}{\partial x_2},\ \frac{\partial}{\partial x_5}\right\}.
\end{equation}
Fix arbitrarily $t\in\bbR$, a leaf $\calL$ of $\calF'_t$, and a curve $\gamma(u)$ in $\calL$, with $\gamma(0):=\overline{x}$.
In view of~\eqref{eq:contact_characteristic_distribution_bis}, there exist $a,b\in C^\infty(\bbR)$ uniquely determined by
\begin{equation*}
	\dot\gamma(u)=a(u)\left.\left(\frac{\partial}{\partial x_4}-t\frac{\partial}{\partial x_2}\right)\right|_{\gamma(u)}+b(u)\left.\frac{\partial}{\partial x_5}\right|_{\gamma(u)}.
\end{equation*}
Consequently the curve $\gamma$ is closed iff there exists $u_0>0$ such that
\begin{equation*}
	\int\limits_{0}^{u_0}a(u)du\in 2\pi\bbZ,\quad
	t\int\limits_{0}^{u_0}a(u)du\in 2\pi\bbZ,\quad
	\int\limits_{0}^{u_0}b(u)du\in 2\pi\bbZ.
\end{equation*}
Since $p|_{\calL}\colon\calL\to\bbT^2$ is a covering map, it follows that  $(p|_{\calL})_\ast\colon \pi_1(\calL,\overline{x})\to\pi_1(\bbT^2,p(\overline{x}))$ is a group monomorphism.
In view of the latter, if $\calL$ is diffeomorphic to $\bbT^2$, then there is a closed curve $\gamma$ in $\calL$, with $\gamma(0)=\overline{x}$, such that $(2\pi)^{-1}\int_{0}^{u_0}a(u)du\in\bbZ\setminus{0}$, and so $t$ has to be rational.
Conversely, if $t\notin\bbQ$, then all the leaves of $\calF'_t$ are diffeomorphic to $\bbR\times\bbS^1$, so non-compact, and the precontact manifold $(S,\ker(s_t^\ast\theta)$ is not integral.

The above discussion shows that there exist coisotropic submanifolds of $(M,C)$, arbitrarily close to $S$, which are not integral.
This leads to the following proposition.

\begin{proposition}
	\label{prop:contact_integral_coisotropic_not_invariant}
	Integral coisotropic submanifolds are not stable under small coisotropic deformations.
\end{proposition}

\subsection{Rigidity of compact integral coisotropic submanifolds}
\label{subsec:rigidity}

Let $L\to M$ be a line bundle, and let $\theta$ be an $L$-valued contact form, with associated contact distribution $C$.
Denote by $\varpi$ the corresponding $L$-valued symplectic Atiyah form, and by $J=\{-,-\}$ the corresponding non-degenerate Jacobi structure on $L\to M$.

Let $S$ be a integral regular coisotropic submanifold of $(M,C)$ and let $\underline{\smash\varphi}_t$ be a smooth integral coisotropic deformation of $S$.
Set $S_t:=\underline{\smash\varphi}_t(S)$ and $L_t:=L_{S_t}$.
Then $S_t$ inherits the $L_t$-valued precontact form $\theta_t:=\theta|_{TS_t}$.
Denote by $\varpi_t$ the $L_t$-valued presymplectic Atiyah form corresponding to $\theta_t$, and by $K_t$ the Lie subalgebroid $\ker(\varpi_t^\flat)\subset DL_t$.
The infinitesimal generator of $\underline{\smash\varphi}_t$ is the smooth $1$-parameter family $\dot{\underline{\smash\varphi}}_t\in\Gamma(TM|_{S_t})$ determined by $\frac{d}{dt}(\underline{\smash\varphi}_t^\ast f)=\underline{\smash\varphi}_t^\ast(\dot{\underline{\smash\varphi}}_t f)$, for all $f\in C^\infty(M)$.

In view of Proposition~\ref{prop:integral_precontact_manifold} we can fix a line bundle $L_B\to B$ and a regular line bundle morphism $\pi\colon L_S=L_0\to L_B$, covering a surjective submersion $\underline{\pi}\colon S=S_0\to B$ with connected fibers, such that the contact reduction of $S$ is performed via $\pi$, and so $K_0=\ker(D\pi)$.

Assume that $S$ is compact.
Being only interested in small deformations of $S$, it is possible to further assume from now on that, up to restrict to small enough $t$ if necessary,
\begin{enumerate}[label=\arabic*)]
	\item
	\label{enumitem:rigidity_assumption_1}
	$\tau\colon M\to S$ is a tubular neighborhood of $S$ in $M$, with $L=\tau^\ast L_0$, and $\underline{\smash\varphi}_t$ is a section of $\tau$,
	\item
	\label{enumitem:rigidity_assumption_2}
	there is a smooth $1$-parameter family of line bundle isomorphisms $\varphi_t\colon L_0\to L_t$ covering $\underline{\smash\varphi}_t\colon S\to S_t$, with $\varphi_0=\id_{L_0}$ such that $(D\varphi_t)K_0=K_t$, 
	i.e.~the contact reduction of $S_t$ is performed via $\pi\circ\varphi_t^{-1}$.
\end{enumerate}

\begin{remark}
	\label{rem:further_assumptions}
	Assumptions~\ref{enumitem:rigidity_assumption_1} and~\ref{enumitem:rigidity_assumption_2} are made possible not only because $S$ is compact but also because its smooth coisotropic deformation $\underline{\smash\varphi}_t$ is integral.
	Specifically, the existence of a (non-unique) $\varphi_t$ as in \ref{enumitem:rigidity_assumption_2} depends on both $S$ being compact and the fact that $\underline{\smash\varphi}_t:S\to S_t$ induces a diffeomorphism of the corresponding characteristic leaf spaces.
	This is precisely the reason  why the current argument holds for smooth coisotropic deformations which are  integral but fails for arbitrary (non-necessarily integral) ones.
\end{remark}

The infinitesimal generator of $\varphi_t$ is the smooth $1$-parameter family $\dot{\varphi}_t\in\Gamma((DL)|_{S_t})$, with $\sigma(\dot{\varphi}_t)=\dot{\underline{\smash\varphi}}_t$, determined by $\frac{d}{dt}(\varphi_t^\ast \lambda)=\varphi_t^\ast(\dot{\varphi}_t\lambda)$, for all $\lambda\in\Gamma(L)$.
Up to quotienting those integral coisotropic deformations which are trivial, $\dot\varphi_t$ comes faithfully encoded into the smooth $1$-parameter family of $L_0$-valued Atiyah $1$-forms $\beta_t:=\varphi_t^\ast(\iota(\dot\varphi_t)\varpi)$.
It is easy to see that $d_D\beta_t=\frac{d}{dt}\varphi_t^\ast \varpi$.
Hence, because of the above assumption~\ref{enumitem:rigidity_assumption_2}, 
\begin{equation}
	\label{eq:infinitesimal_integral_coisotropic_deformations}
	\iota(\square)d_D\beta_t=0,\qquad\calL_\square d_D\beta_t=0,
\end{equation}
for all $\square\in\Gamma(K_0)=\Gamma(\ker(D\pi))\subset\calD L_0$.
As a consequence, in view of Equation~\eqref{eq:contracting_homotopy} and Lemma~\ref{lem:basic_Atiyah_forms}, there is a smooth path $\lambda_t$ in $\Gamma(L_0)$ such that  $\beta_t-d_D\lambda_t\in\pi^\ast(\Omega_{L_B}^1)$, and so, a fortiori, $\beta_t-d_D\lambda_t\in\Omega^1_{L_0}$ annihilates $K_0\subset DL_0$.
The above assumption~\ref{enumitem:rigidity_assumption_1} guarantees that there is a smooth path $\tilde\lambda_t$ in $\Gamma(L)$ such that $\lambda_t=\varphi_t^\ast\tilde\lambda_t$, and so also $\iota(\dot\varphi_t)\varpi-d_D(\tilde\lambda_t|_{S_t})\in\Omega^1_{L_t}$ annihilates $K_t\subset DL_t$.
Since $K_t^{\perp\varpi}=DL_{t}$ (cf.~Proposition~\ref{prop:regular_coisotropic_submanifolds}), the latter can be equivalently rewritten as
\begin{equation}
	\label{eq:rigidity_final}
	\dot\varphi_t\equiv (\varpi^\flat)^{-1}(d_D\tilde\lambda_t)|_{S_t}\mod\calD L_{t}.
\end{equation}
Recall that $\Delta_{\tilde\lambda_t}:=(\varpi^\flat)^{-1}(d_D\tilde\lambda_t)=J^\sharp(j^1\tilde\lambda_t)=\{\tilde\lambda_t,-\}$ is a time-dependent Hamiltonian derivation of the Jacobi bundle $(L,J)$ (cf., e.g.,~\cite{LOTV,LTV,T2017}).
Hence it generates a smooth $1$-parameter family of local automorphisms $\psi_t$ of the Jacobi bundle $(L,J)$ which covers the contact isotopy $\underline{\smash\psi}_t$ of $(M,C)$ generated by the time dependent contact vector field $\calX_{\tilde\lambda_t}=\sigma(\Delta_{\tilde\lambda_t})$. 
Finally, from what above and~\eqref{eq:rigidity_final}, it follows that $S_t:=\underline{\smash\varphi}_t(S)$ coincides with $\underline{\smash\psi}_t(S)$.

The above discussion shows that locally every smooth integral coisotropic deformation of $S$ is induced by a contact isotopy of $(M,C)$.
Since the remaining Legendrian case is already well-known, this leads to the following.

\begin{theorem}
	\label{prop:rigidity_compact_integral_coisotropic}
	Compact integral coisotropic submanifolds are rigid, i.e.~locally their smooth integral co\-iso\-tro\-pic deformations are induced by contact isotopies.
\end{theorem}

As a by-product of the above discussion we also get the following.

\begin{proposition}
	\label{prop:unobstructedness}
	The (compact) integral coisotropic deformation problem, under contact equivalence, is unobstructed and its moduli space is discrete.
\end{proposition}

\begin{proof}
	From~\eqref{eq:infinitesimal_integral_coisotropic_deformations} it follows that the infinitesimal integral coisotropic deformations of $S$ are encoded into the vector subspace $\frakB_S\subset\Gamma((DL)|_S)/\calD L_S$ formed by equivalence classes modulo $\calD L_S$ of sections $\square\in\Gamma((DL)|_S)$ such that $d_D\beta\in\pi^\ast(\Omega^2_{L_B})$, where $\beta:=i_S^\ast(\iota(\square)\varpi)\in\Omega^1_{L_S}$.
	The same argument used to get~\eqref{eq:rigidity_final} implies now that, modulo $\calD L_S$, any $\square$ of this kind agrees with an Hamiltonian derivation $\Delta_{\tilde\lambda}$ of $(J,L)$, for some $\tilde\lambda\in\Gamma(L)$.
\end{proof}

\begin{remark}
	It is possible to compare Proposition~\ref{prop:unobstructedness} with the analogue result in the symplectic case (cf.~\cite{ruan2005deformation}).
	On the symplectic side, for a compact integral coisotropic submanifold $S$ with characteristic foliation $\calF$, its local moduli space under Hamiltonian equivalence consists of the elements of $H^1(\calF)$ that, seen as sections of a vector bundle over $B=S/\calF$, are flat wrt the Gauss--Manin connection.
	Similarly on the contact side, for a compact integral coisotropic submanifold $S$ with characteristic foliation $\calF$, the local moduli space of $S$ under contact equivalence consists of the elements of $H^1(\calF;L_S)$ that, seen as sections of a vector bundle over $B$, are flat wrt a certain connection along $DL_B$, but now it turns out that there are no non-zero flat sections.
\end{remark}

\section*{Acknowledgements}
	The author is grateful to A\"issa Wade for her help with an earlier version of this note, and to Luca Vitagliano and Marco Zambon for their comments and suggestions.



\end{document}